\documentclass[12pt,bezier]{article}
\usepackage{amssymb}
\usepackage{mathrsfs}
\usepackage{amsfonts}
\usepackage{graphics}
\usepackage{enumerate}
\usepackage{xcolor}
\usepackage{graphicx}
\usepackage{subfig}
\usepackage{float}
\usepackage{amsmath}
\usepackage{amsfonts,amsthm,amssymb}
\newtheorem{lem}{Lemma}
\newtheorem{thm}[lem]{Theorem}
\newtheorem{cor}[lem]{Corollary}

\newtheorem{clm}{Claim}

\newtheorem{con}[lem]{Conjecture}

\begin{document}
\title{ On regular 2-path Hamiltonian graphs }
\author{Xia Li$^a$, \quad Weihua Yang$^a$\footnote{Corresponding author. E-mail: ywh222@163.com; yangweihua@tyut.edu.cn
(W.~Yang).}, \quad Bo Zhang$^b$,\quad Shuang Zhao$^a$ \\
\small $^a$Department of Mathematics,  Taiyuan University of
Technology,  Taiyuan 030024,  China\\
\small $^b$Department of Mathematics,  Shanxi Normal University, Jinzhong 030600,  China}
\maketitle {\flushleft\bf Abstract:} {\small Kronk introduced the $l$-path hamiltonianicity of graphs in 1969. A graph is $l$-path Hamiltonian if every path of length not exceeding $l$ is contained in a Hamiltonian cycle. We have shown  that if $P=uvz$ is a 2-path of a 2-connected, $k$-regular graph on at most $2k$ vertices and $G - V(P)$ is connected, then there must exist a Hamiltonian cycle in $G$ that contains the 2-path $P$.
In this paper, we characterize a class of graphs that illustrate the sharpness of the bound $2k$. Additionally, we show that by excluding the class of graphs, both 2-connected, $k$-regular graphs on at most $2k + 1$ vertices and 3-connected, $k$-regular graphs on at most $3k-6$ vertices satisfy that there is a Hamiltonian cycle containing the 2-path $P$ if $G\setminus V(P)$ is connected.
 }


\section{Introduction}

A \emph{Hamiltonian path $($cycle$)$} in a graph $G$ is a path (cycle) containing all the vertices of $G$, and a graph with a Hamiltonian cycle is called
\emph{Hamiltonian}. A graph is \emph{Hamilton-connected} when every pair of
distinct vertices is connected by a Hamiltonian path.
The existence of Hamiltonian cycles in 2-connected, $k$-regular graphs has been the subject of research in several publications~\cite{EH,BH,J,ZLY,BK,L}. The findings of these studies indicate that all 2-connected, $k$-regular graphs on at most $3k + 4$ vertices, except for two kinds of graphs which is not 3-connected, are Hamiltonian. In 1976, H\"{a}ggkvist proposed the following conjecture .

\begin{con}[\cite{H}]\label{thm7}
For $k\geqslant 4$, every $m$-connected, $k$-regular graph on at most $(m+1)k$ vertices is Hamiltonian.
\end{con}

The example constructed independently by Jackson and Jung (refer to~\cite{Bill}) provides evidence to refute Conjecture~\ref{thm7} for $m \geqslant4$. However, there exists a remaining unresolved case when $m=3$.
\begin{con}[refer to~\cite{Bill}]\label{thm6}
For $k\geqslant 4$, every $3$-connected, $k$-regular graph on at most $4k$ vertices is Hamiltonian.
\end{con}
Considerable progress has been made in investigating the existence of Hamiltonian cycles within $3$-connected, $k$-regular graphs. Conjecture~\ref{thm6} has been resolved for significantly large graphs according to~\cite{Kuhn}, but for a very large (albeit finite) number of cases it remains open. In addition, the proof provided in~\cite{Kuhn} is extensive and intricate, making a simpler proof highly desirable.

Another noteworthy subarea within Hamiltonian graph theory focuses on Hamiltonian cycles that contain specified elements of a graph. Examples include $k$-ordered Hamiltonian graphs~\cite{NS,MWY}, edge-Hamiltonian graphs~\cite{Hao Li}, and others.
One of these directions is the study of $l$-path Hamiltonicity. A graph $G$ on $n$ vertices is said to be \emph{$l$-path Hamiltonian} if every path of length not exceeding $l$, $1\leqslant l\leqslant n-2$, is contained in a Hamiltonian cycle. Kronk in~\cite{Hudson V. Kronk} proved that for a graph $G$ on $n$ vertices, if the sum of the degrees of every pair of non-adjacent vertices of $G$ is at least $n+l$, where $l$ is a positive integer, then $G$ is $l$-path Hamiltonian. The idea of combining $l$-path Hamiltonicity with regular graphs comes from the following result proved by Li in~\cite{Hao Li}.

\begin{thm}[\cite{Hao Li}]\label{thm1}
Let $G$ be a $2$-connected, $k$-regular graph on at most $3k-1$ vertices, and let $e=uv$ be any edge of $G$ such that $\{u, v\}$ is not a cut-set. Then $G$ has a Hamiltonian cycle containing $e$.
\end{thm}

Theorem~\ref{thm1} shows that a 2-connected, $k$-regular graph $G$ on at most $3k-1$ vertices satisfying $G-V(P)$ is connected for every path $P$ of length 1 is 1-path Hamiltonian. Motivated by above result, Li and Yang in~\cite{Xia Li} proved the following result.

\begin{thm}[\cite{Xia Li}]\label{thm2}
Let $G$ be a $2$-connected, $k$-regular graph on at most $2k$ vertices, and let $P=uvz$ be any path of  $G$ such that $\{u, v, z\}$ is not a cut-set. Then $G$ has a Hamiltonian cycle containing $P$.
\end{thm}

From Theorem~\ref{thm1} and Theorem~\ref{thm2}, it can be deduced that if a 2-connected, $k$-regular graph $G$ on at most $2k$ vertices has the property that the graph $G-V(P)$ is connected for every path $P$ of length at most 2, then $G$ is 2-path Hamiltonian.

For positive integer $q\geqslant k$, we define that a class $\mathscr{H}$ of graphs of path $P=uvz$ is a 2-connected, $k$-regular graph on $2q+1$ vertices, which contains two disjoint sets $X$ and $Y$ of vertices such that $Y$ is independent, $X$ contains $\{u,v,z\}$, $|Y|=q$, $|X|=q+1$, $N(Y)\subseteq X$ and $v$ is adjacent to $u$ and $z$.  Figure 1 below illustrates a 2-connected, $k$-regular graph on $2k+1$ vertices belonging to $\mathscr{H}$ with $q=k=4$.

\begin{figure}[H]
  \centering
  \includegraphics[]{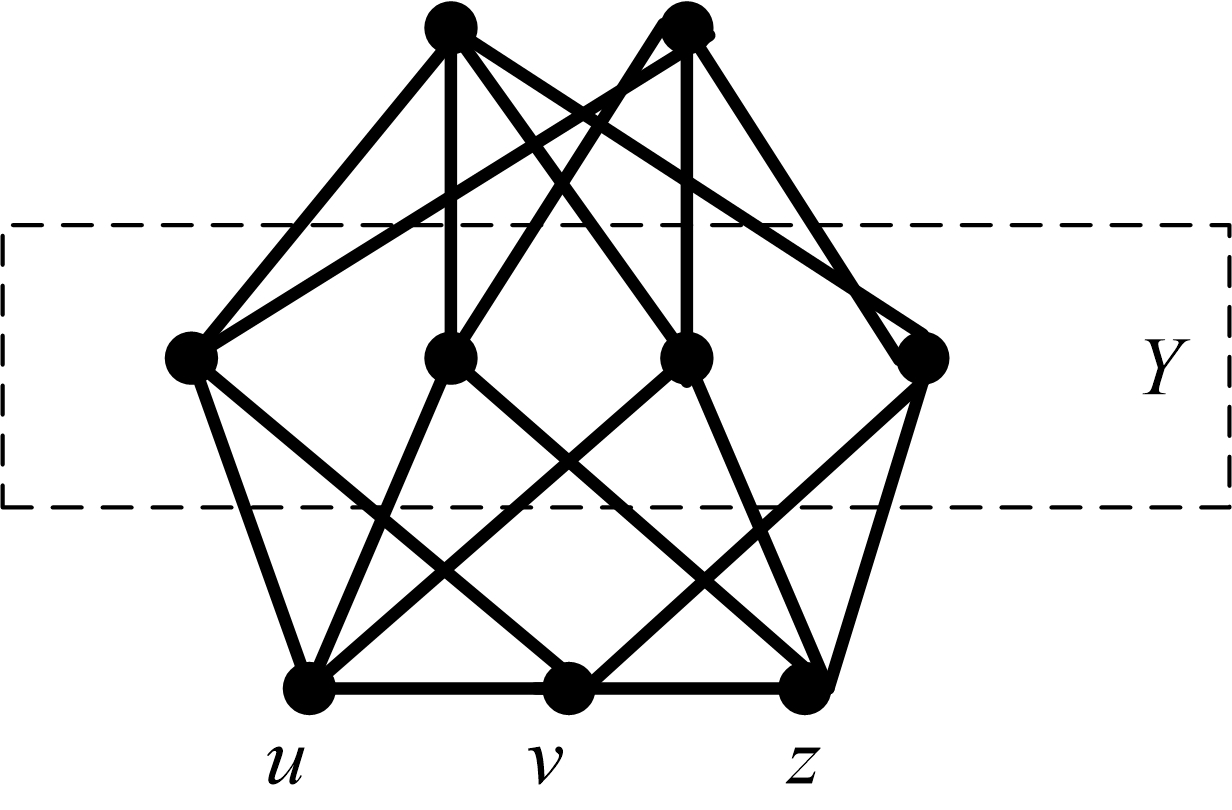}\\
  \caption{}
\end{figure}

The existence of the graph class $\mathscr{H}$ demonstrates that the upper bound $2k$ in Theorem~\ref{thm2} is sharp. Clearly, when $k$ is odd, the class of graphs $\mathscr{H}$ does not exist. Therefore, it is worthwhile to investigate whether the result in Theorem~\ref{thm2} can be strengthened by excluding the graph class $\mathscr{H}$.
In this paper, we are going to prove the following.

\begin{thm}\label{thm3}
Let $G$ be a $2$-connected, $k$-regular graph on $n$ vertices, and let $P=uvz$ be any path of $G$ such that $\{u, v, z\}$ is not a cut-set. If $G\notin \mathscr{H}$ and $n\leqslant 2k+1$, then $G$ has a Hamiltonian cycle containing $P$.
\end{thm}

The following corollary  follows from Theorem~\ref{thm1} and Theorem~\ref{thm3}.

\begin{cor}\label{thm4}
Let $G$ be a $2$-connected, $k$-regular graph on at most $2k+1$ vertices. If $G-V(P)$ is connected for every path $P$ of length at most $2$ and $G\notin \mathscr{H}$, then $G$ is $2$-path Hamiltonian.
\end{cor}
We present an illustrative example that demonstrates the sharpness of the bound in Corollary~\ref{thm4}. Let $H_{i}$, $i\in\{1,2\}$, be a graph which is obtained from $K_{k+1}$ by deleting one edge $e_{i}=a_{i}b_{i}$. We can construct a 2-connected, $k$-regular graph $G$ on $2k+2$ vertices from $H_{1}$ and $H_{2}$ by adding $a_{1}a_{2}$ and $b_{1}b_{2}$. Notably, there is a 2-path in $G$ that is not contained in any Hamiltonian cycle of $G$. Obviously, this example is not 3-connected graph. Is it possible to improve the result for 3-connected, $k$-regular graphs? In this paper, we also prove the following.

\begin{thm}\label{thm5}
Let $G$ be a $3$-connected, $k$-regular graph on $n$ vertices, and let $P=uvz$ be any path of  $G$ such that $\{u, v, z\}$ is not a cut-set. If $G\notin \mathscr{H}$ and $n\leqslant 3k-6$, then $G$ has a Hamiltonian cycle containing $P$.
\end{thm}

The following corollary  follows from Theorem~\ref{thm1} and Theorem~\ref{thm5}.

\begin{cor}\label{thm8}
Let $G$ be a $3$-connected, $k$-regular graph on at most $3k-6$ vertices. If $G - V(P)$ is connected for every path $P$ of length at most $2$ and $G\notin \mathscr{H}$, then $G$ is $2$-path Hamiltonian.
\end{cor}

The $L$-graph (as discussed in~\cite{Hao Li}) as a counterexample of Theorem~\ref{thm1} with $3k$ vertices is 3-connected for $k\geqslant6$.  Consequently, the bound provided in Corollary~\ref{thm8} is nearly optimal.
\section{Notation and preliminaries}
All graphs mentioned in this paper are finite simple  graphs. For standard graph theory notation and terminology not explained in this paper, we refer the reader to~\cite{bondy}. Let $G$ be a graph. $|G|$ and $\delta$ denote the number of vertices and the minimum degree of $G$, respectively. For $x \in V(G)$ and $S\subseteq V(G)$, $N_{G}(x)$ denotes the neighbors of $x$ in $G$, $N_{G}(S)=\bigcup_{x\in S} N_{G}(x)$ and $d_{G}(x)=|N_{G}(x)|$. For a cycle $C$ in $G$ with a fixed orientation, and any two vertices $x, y$ on $C$, we denote by $x^{+}$ and $x^{-}$ the following vertex and the preceding vertex of $x$ according to the orientation of $C$, respectively. We define the segment $C [x, y]$ to be the set of vertices on $C$ from $x$ to $y$ (including $x$ and $y$) according to the orientation and let $C(x,y) = C[x,y] - \{x,y\}$. Analogously, $C[x,y)$ and $C(x,y]$ are also defined. Let $A$ be a set of vertices of $G$. An $A$-segment is a $C[x, y]$ segment such that $C [x, y] \cap A = \{x,y\}$. We put $x^{+2}=(x^{+})^{+}$ ($x^{-2}=(x^{-})^{-})$ and $x^{+i}=(x^{+(i-1)})^{+}$ ($x^{-i}=(x^{-(i-1)})^{-})$.

For a cycle $C$ and any $A, B\subseteq V(G)$, let $E(A, B)=\{uv\in E(G):u\in A, \ v\in B\}$ ($E^{'}(A, B)=\{uv\in E(G)-E(C):u\in A, v\in B\}$) and $E(A)=\{uv\in E(G):u, v\in A\}$ ($E^{'}(A)=\{uv\in E(G)-E(C):u, v\in A\}$). Put $e(A, B)=|E(A, B)|$, $e^{'}(A, B)=|E^{'}(A, B)|$, $e(A)=|E(A)|$ and $e^{'}(A)=|E^{'}(A)|$. For convenience, we often use a subgraph $H$ of a graph $G$ to denote its vertex set $V(H)$ if no confusion
arises. We shall use the following result.

\begin{thm}[\cite{Bill}]\label{thm9}
Let $G$ be a connected graph such that for every longest path $P$ in $G$, the sum of the degrees of the end-vertices of $P$ is at least $|P|+1$. Then $G$ is Hamilton-connected.
\end{thm}
We first introduce an \emph{operation} used during the proofs of Theorems~\ref{thm3} and Theorems~\ref{thm5}. Let $G$  be a connected graph on $n$ vertices, and let $P=uvz$ be a path of  $G$. We define a new graph $G_{1}$ by inserting two vertices $w_{1}$ and $w_{2}$ on the edges $e_{1}=uv$ and $e_{2}=vz$ of $P$ respectively. Then we have $G_{1}=(G-\{e_{1}, e_{2}\})\cup\{w_{1}, w_{2}\}\cup\{uw_{1}, w_{1}v, vw_{2}, w_{2}z\}$, $P_{1}=u w_{1} v w_{2} z$ and $|G_{1}|=n_{1}=n+2$. Clearly, if we want to prove that $P$ is contained in a Hamiltonian cycle in $G$, it is sufficient to prove that $G_{1}$ is Hamiltonian.
\section{Proof of Theorem~\ref{thm3}}
From Theorem~\ref{thm2}, we have that Theorem~\ref{thm3} holds for $n\leqslant 2k$. So we only need to consider the case $n=2k+1$. Let $G$  be a 2-connected, $k$-regular graph on $n=2k+1$ vertices, and let $P=uvz$ be a path of  $G$ such that $\{u, v, z\}$ is not a cut-set. After the operation in section 2, we have $|G_{1}|=n_{1}=2k+3$. Suppose $G_{1}$ is not Hamiltonian. Let $C$ be a longest cycle of $G_{1}$ containing $w_{1}$ and $w_{2}$, such that the number of components of $R=G_{1}-C$ is as small as possible. Let $r=|R|$ and $C=c_{1} c_{2}  \cdots  c_{n_{1}-r}$. The subscripts of $c_{i}$ will be reduced modulo $n_{1}-r$ throughout. Obviously, we have $|C|=n_{1}-r \geqslant 6$.

Suppose $R$ is an independent set, then let $v_{0}$ be an isolated vertex in $R$. Put $Y_{0}=\emptyset$, and for any $j \geqslant 1$, $X_{j}=N(Y_{j-1}\cup\{v_{0}\})$, $Y_{j}=\{c_{i}\in V(C):c_{i-1}, c_{i+1}\in X_{j}\}$, $X=\bigcup\limits_{j = 1}^\infty  {{X_j}}$, $Y=\bigcup\limits_{j = 0}^\infty  {{Y_j}}$, $x=|X|\geqslant k $ and $y=|Y|$ .
By the hopping lemma (\cite{Woodall}), we have $X \subset V(C),\  X \cap Y = \emptyset $ and $X$ does not contain two consecutive vertices of $C$. Let $S_{1}, S_{2},  \cdots , S_{x}$ be the sets of vertices contained in the open $X$-segment of $C$ (the sets of vertices on $C$ between $X$ satisfying $S_{i}\cap X = \emptyset$ for each $i$ with $1\leqslant i \leqslant x$). Put $\phi = \{ {S_i}:\left| {{S_i}} \right| \geqslant2, 1\leqslant i\leqslant x\}.  $
Then $S_{i}=\{c_{l}, c_{l+1}, \cdots , c_{m}\}\in\phi $ is said to be \emph{$\psi $-connected} to $S_{j}=\{c_{q}, c_{q+1}, \cdots , c_{z}\}\in\phi $ if $\left| {{S_i}} \right|$ is odd and $c_{q}$ and $c_{z}$ are both joined to $c_{l+e}$ for all odd $e$, $1\leqslant e \leqslant m-l-1$. Now, $c_{l+1}, c_{l+3}, \cdots, c_{m-1}$ are called \emph{$P$-vertices} of $S_{i}$.
Set $P\  =\  \{c_{i}\in V(C)$: $c_{i}$ is a $P$-vertex of some $S_{j}$ which is $\psi$-connected to some $S_{t}$ of $\phi$\}, and $p=|V(P)|$.
By the same proof as the case 1 in \cite{Xia Li}, we have the  following  inequality which is the inequality (4) in \cite{Xia Li},
\begin{equation}\label{eq1}
p + 4 \leqslant (n_{1} - 2x - k)(n_{1} - 1 - 2x - p) + k - 2(r_{1} - 1)(x - y - 1).
\end{equation}
From the definition of $P$, we have  $p \leqslant \frac{n_{1} - 1 - 2x}2$, which implies $n_{1} - 1 - 2x -p\geqslant 2p - p \geqslant 0$. And $k \geqslant n_{1} - 1- 2x - p$ by $n_{1}= 2k+3$ and $x\geqslant k$. So we have
\begin{equation}\label{eq2}
p + 4 \leqslant (n_{1} - 2x - k + 1)k - 2(r_{1} - 1)(x - y - 1).
\end{equation}
By the definitions of $X$ and $Y$, we have $x\geqslant y$. Now, we claim $x\geqslant y+1$. Otherwise, when $x=y$, since $d_{G_{1}}(w_{1})=d_{G_{1}}(w_{2})=2$, we have $u, v, z$ belong to $X$ and $w_{1}$, $w_{2}$ belong to $Y$. If $R$ contains at least two isolated vertices, then $|Y\cup R|\geqslant y+2=x+2$. Because $Y\cup R$ is an independent set and $N(Y\cup R)\subseteq X$, we have that there are at least $ky+4=kx+4$ edges from $Y\cup R$ to $X$, but $X$ accepts at most $kx$ edges, a contradiction. So, when $x=y$,  $R$ contains only one isolated vertex, which implies $n_{1}=2x+1$ and $n=2x-1=2q+1$ for $q=x-1$. By definition of $\mathscr{H}$, we have  $G\in \mathscr{H}$, a contradiction.
Therefore by (\ref{eq2}), we have $n_{1} - 2x - k + 1 > 0$, and then $n_{1} > 3k-1\geqslant 3k$, which contradicts $n_{1} =2k+3$.

Thus in the following proof, we assume that there exists a component $H$ in $R$ such that $|H| \geqslant 2$. For a path $Q = {q_1} {q_2}  \cdots  {q_g}, g \geqslant 2$, in $H$, let $t(Q)$ denote the number of $C[c_{i}, c_{j}]$  such that $c_{i}$ is joined to one of $q_{1}$ and $q_{g}$, $c_{j}$ is joined to the other, and $e(\left\{ {{q_1}, {q_g}} \right\}, \{{c_{i + 1}}, {c_{i + 2}},  \cdots , {c_{j - 1}}\}) = 0$. We say that $Q$ satisfies the condition $(*)$ if $t(Q)\geqslant 2$,  ${N_{C}}(\left\{ {{q_1}, {q_g}} \right\}) \not\subset
 \left\{ {u, v, z} \right\}$ and there is a  $C[c_{i}, c_{j}]$ such that $u, v, z, {w_1}\ $and$ \ {w_2} \notin \left\{ {{c_{i + 1}}, {c_{i + 2}},  \cdots , {c_{j - 1}}} \right\}$. Now, let $H$ be the largest component of $R$ and $h=|H|$. The rest of the proof of Theorem~\ref{thm3} is divided into two cases. We first consider the case of $k\geqslant6$, and we prove it in the following two cases.

$\bf Case \  1. $ $2 \leqslant h \leqslant k$.

\begin{clm}\label{lem1}
For $k\geqslant6$, if $2 \leqslant h \leqslant k$, then $H$ is Hamilton-connected.
\end{clm}

\begin{proof}
By contradiction, suppose $H$ is not Hamilton-connected. From Theorem~\ref{thm9}, we can choose a longest path $Q = {q_1} {q_2}  \cdots  {q_g}, g \geqslant 2$, in $H$, satisfying ${d_{H}}({q_1}) + {d_{H}}({q_g}) \leqslant |Q|=g $. For any $v\in V(H)$, since $h\leqslant k$, we have ${N_{C}}({v}) \geqslant 1$. Let $X={N_{C}}({q_1}) \cap {N_{C}}({q_g})$ and $|X|=x$.

First, we prove that ${N_{C}}({q_1}) = {N_{C}}({q_g})$. Otherwise, without loss of generality assume that ${d_{C}}({q_1}) \leqslant {d_{C}}({q_g})$. If ${d_{C}}({q_1}) = 1$, we have ${d_{C}}({q_g}) \geqslant  2k-g-1$ since ${d_{C}}({q_1}) + {d_{C}}({q_g}) \geqslant 2k-g $. Since $C$ is the longest cycle of $G_{1}$ containing $w_{1}$ and $w_{2}$, we have that every ${N_{C}}({q_g})$-segment of $C$ contains at least one  interior vertex. Hence $ n_{1} \geqslant |H| +|C| \geqslant g + 2{d_{C}}({q_g})\geqslant g + 2(2k-g-1)=4k-g-2 \geqslant 3k-2$, a contradiction. Thus, we have $2 \leqslant {d_{C}}({q_1}) \leqslant {d_{C}}({q_g})$. It is easy to prove that $t(Q)\geqslant 2$ and $x+1\leqslant t(Q)$.
Since ${N_{C}}(\{q_1,\ q_g\})\cup ({N_{C}}(\{q_1,\ q_g\}))^{+}\cup H \subseteq V(G_{1}) $, we have
\[\begin{array}{l}
n_{1} \geqslant \left| C \right| + \left| H\right| \geqslant \left| H \right|+2|{N_{C}}(\{q_1,\ q_g\})|+(t(Q)-2)(g-1)\\
 \ \  \ \ \geqslant g+2[{d_{C}}({q_1}) + {d_{C}}({q_g})]-2|{N_{C}}({q_1}) \cap {N_{C}}({q_g})|+(t(Q)-2)(g-1)\\
 \ \ \ \ \geqslant g+2(2k-g)-2x+(t(Q)-2)(g-1)\\
 \ \ \ \ \geqslant 4k-g-2(x+1)+2+(t(Q)-2)(g-1)\\
 \ \ \ \ \geqslant 4k-g-2t(Q)+2+(t(Q)-2)(g-1)\\
 \ \ \ \ \geqslant 3k-2+(t(Q)-2)(g-3).
\end{array}\]
Since H is not Hamilton-connected, we have $g\geqslant3$. This implies $(t(Q)-2)(g-3)\geqslant0$. Therefore, we have $n_{1} \geqslant 3k - 2$, a contradiction.

From the above discussion, we have ${N_{C}}({q_1}) = {N_{C}}({q_g})=X$ and $t(Q)= x={d_{C}}({q_1})={d_{C}}({q_g})\geqslant k-\frac{g}2\geqslant\frac{k}2\geqslant3$. Then
\[\begin{array}{l}
n_{1} \geqslant \left| C \right| + \left| H \right| \geqslant (t(Q)-2)g+2+t(Q)+h\geqslant (x-2)g+2+x+g\\
 \ \ \ \ \geqslant (x-1)(g+1)+3\geqslant (k-\frac{g}2-1)(g+1)+3.
\end{array}\]

Since $f(g)=(k-\frac{g}2-1)(g+1)+3$ is a concave function of $g$, $3\leqslant g\leqslant k$, we have $f(3) = 4k-7 > 2k+3$ and $f(k)=\frac{k^{2}}2-\frac{k}2+2  > 2k+3 $ when $k\geqslant 6$. Hence $f(g)> 2k+3$, a contradiction.

\end{proof}
$\bf Subcase \  1.1. $ $h=k$.

If there is a $N_{C}(H)$-segment $C[c_{i}, c_{j}]$ such that $u, v, z, {w_1}\ $and$ \ {w_2} \notin \left\{ {{c_{i + 1}}, {c_{i + 2}},  \cdots , {c_{j - 1}}} \right\}$, we have $|C(c_{i}, c_{j})|\geqslant h$ and $ n_{1} \geqslant|H| +|C| \geqslant |H| +|C(c_{i}, c_{j})|+|\{v, w_{1}, w_{2}, c_{i}, c_{j}\}|\geqslant h+h+5=2k+5$, a contradiction. Thus, $|N_{C}(H)|=2$ and $v\in N_{C}(H)$. Let $X=N_{C}(H)$. Since $h=k$, for any $v_{i}\in H, \ i\in\{1 , \cdots, h \}$, we have $|{N_{C}}({v_i})| \geqslant 1$. So $e(H, X)\geqslant k$.
Because $|G_{1}-H-X|= n_{1}-k-2=2k+3-k-2=k+1$, we have
$$e(G_{1}-H-X) = e(G_{1}-H-X-\{w_{1}, w_{2}\})+e(G_{1}-H-X, \{w_{1}, w_{2}\})\leqslant\frac{(k-2)(k-1)}2+2.$$
Since $G$ is a $k$-regular graph, we have $$e(G_{1}-H-X, X) =k|G_{1}-H-X-\{w_{1}, w_{2}\}|+4-2e(G_{1}-H-X)-e(G_{1}-H-X, H).$$
Since $e(G_{1}-H-X, H)=0$, we have $$e(G_{1}-H-X, X) \geqslant k(k-1)+4-2(\frac{(k-2)(k-1)}2+2)=2k-2\geqslant k+1.$$
On the other hand $e(G_{1}-H-X, X)\leqslant e(G_{1}, X)-e(H, X)\leqslant 2k-k=k$, a contradiction.

$\bf Subcase \  1.2. $ $\frac{k+1}2 \leqslant h \leqslant k-1$.

Since $H$ is Hamilton-connected, it is easy to deduce that there exists a Hamiltonian path $Q$ in $H$ such that $Q$ satisfies $(*)$. By a similar proof to Lemma 6 in~\cite{Xia Li}, we have the following claim.

\begin{clm}\label{lem2}
There exists a Hamiltonian path $Q$ in $H$ such that $t(Q)\geqslant 3$.
\end{clm}

Since $t(Q)\geqslant 3$, there is a  $C[c_{i}, c_{j}]$ such that $u, v, z, {w_1}\ $and$ \ {w_2} \notin \left\{ {{c_{i + 1}}, {c_{i + 2}},  \cdots , {c_{j - 1}}} \right\}$. Then there exists either ${c_{i}^{-}} \notin \{ {w_1}, {w_2}\}$ or ${c_{j}^{+}} \notin \{ {w_1}, {w_2}\}$. Without loss of generality, let ${c_{i}^{-}} \notin \{ {w_1}, {w_2}\}$. Since $C$ is a longest cycle of $G_{1}$ containing $w_{1}$ and $w_{2}$, we have ${N_{G_1}}({c_{i}^{-}}) \cap \left[ {H \cup {{c_{j - 1}}, {c_{j- 2}},  \cdots , {c_{j- h}}}\cup \left\{ {c_{i}^{-}} \right\}} \right] = \emptyset $. And there are at least two of $\{v, w_{1}, w_{2}\}$ which can not be adjacent to ${c_{i}^{-}}$. This implies ${d_{G_1}}({c_{i}^{-}}) \leqslant 2k + 3 - ( {h + h + 3} ) \leqslant k - 1$, a contradiction to ${d_{G_1}}({c_{i}^{-}}) = k$.

$\bf Subcase \  1.3. $ $2 \leqslant h \leqslant \frac{k}2$.

For any $v\in V(H)$, since $2 \leqslant h\leqslant \frac{k}2$, we have ${N_{C}}({v}) \geqslant k-h+1\geqslant k-\frac{k}2+1\geqslant4$. So $N_{C}(H)\geqslant k-h+1$. Let $a \in V(H)$ and ${N_{C}}({a})=A$.

 \begin{clm}\label{lem3}
For any $A$-segment $C [c_{i}, c_{j}]$ satisfying $C [c_{i}, c_{j}] \cap \{w_{1},w_{2}\}=\emptyset$, we have $C [c_{i}, c_{j}] \cap N_{C}((V(H)-a))\neq \emptyset$.
\end{clm}

\begin{proof}
Let $N_{C}((V(H)-a))=S$. Suppose that there is an $A$-segment $C[c_{i}, c_{j}]$ satisfying $C[c_{i}, c_{j}] \cap [\{w_{1},w_{2}\}\cup S]=\emptyset$. Let $S=\left\{ {{c_{{r_1}}}, {c_{{r_2}}},  \cdots , {c_{{r_s}}}} \right\}$, where ${c_{{r_1}}}$ and ${c_{{r_s}}}$ are the closest vertices to $c_{j}$ and $c_{i}$ in $S$, respectively. Clearly, $|S|=s\geqslant k-h+1\geqslant 4$. Therefore, there is at least one segment of $C [c_{j}, c_{r_1}]$ and $C [c_{r_s}, c_{i}]$ which does not contain $w_{1}$ and $ w_{2}$. Without loss of generality, let $w_{1}, \ w_{2} \notin C [c_{r_s}, c_{i}]$. Obviously, we have ${N_{C}}(c_{i}^{+}) \cap \left[ {H \cup \left\{{ c_{r_s+1}, c_{r_s+2},\cdots, c_{r_s+h},{w_1}, {w_2}, {c_{i}^{+}}} \right\}} \right] = \emptyset $.  And for any $C [c_{r_i}, c_{r_{i+1}}] $ satisfying $C [c_{r_i}, c_{r_{i+1}}] \cap \{w_{1},\ w_{2}\}=\emptyset$ , $i\in\{1,2,\cdots,s-1\}$, we have ${N_{C}}(c_{i}^{+}) \cap \{ c_{r_i+1},c_{r_i+2}\} = \emptyset $.
This implies \[\begin{array}{l}
d_{G_{1}}({c_{i}^{+}})\leqslant 2k+3 -[ h + 2(s-2-1)+h+3]  \ \ \ \ \\
 \ \ \ \ \ \ \ \ \ \ \    \leqslant2k+3 -[ h + 2(k-h+1-2-1) + h+3]\leqslant 4,
\end{array}\]
a contradiction.
\end{proof}
By Claim~\ref{lem3}, we have
\begin{equation}\label{eq3}
   \left| C \right| \geqslant |A|+(|A|-2)h+2\geqslant (k-h+1)+(k-h+1-2)h+2
\end{equation}
and  $n_{1}\geqslant \left| C \right| + \left| H \right|\geqslant(k-h+1)+(k-h+1-2)h+2+h=k+3+(k-h-1)h$.

Put $g(h)=k+3+(k-h-1)h$. For $k\geqslant7$, since $g(h)$ is a concave function of $h$ with $g(2)= 3k-3> 2k+3$ and $g(\frac{k}2)=\frac{k^{2}}4+\frac{k}2+3> 2k+3$. Hence $g(h)> 2k+3$, a contradiction. When $k=6$ and $n_{1}=15$, there are two cases where $h=2$ and $h=3$ to consider as follows.

Case (a): $h=2$. By (\ref{eq3}), we have $13=n_{1}-h\geqslant|C|\geqslant (k-h+1)+(k-h+1-2)h+2=13$. So $R$ has only one component $H$ and $|C|=|A|+(|A|-2)h+2=13$. Since $C$ is a longest cycle of $G_{1}$ containing $w_{1}$ and $w_{2}$, we have $N_{C}(v_{i})=N_{C}(v_{j}), i\neq j$, for any $v_{i}\in H, i\in \{1, 2\}$. Let $N_{C}(H)=X$ and $Z=X^{+}\cup X^{-}$. Clearly, we have $u, v, z\in X$ and $e^{'}(Z)=0$, otherwise, there exists a longer cycle containing $w_{1}$ and $w_{2}$. Since $e(Z, H)=0$, $|X|=5$ and $|Z|=8$, we have $e(Z, X)=(6-1)\times6+4=34$. However, $e(Z, X)\leqslant k|X|-e(H, X)\leqslant20$, a contradiction.

Case (b): $h=3$. It is similar to case (a) above.

The following figure 2 shows the edges between $H$ and $C$ in the above two cases.

\begin{figure}[H]
  \centering
  \includegraphics[]{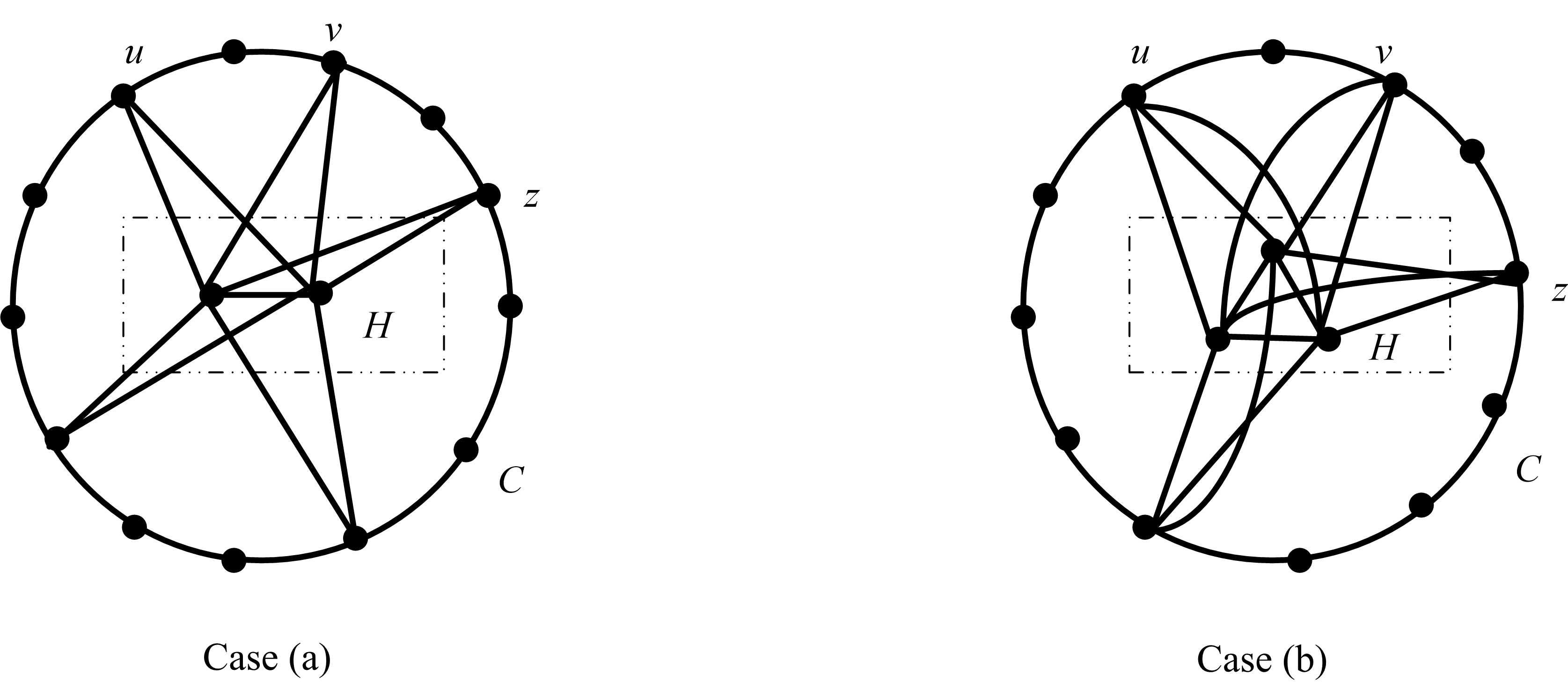}\\
  \caption{}
\end{figure}

$\bf Case \ 2. $ $h\geqslant k+1. $

By the assumption of connectivity and as $\{u, v, z\}$ is not a cut-set, there exists $x{''} \in {N_{C}}(H)$, such that $x^{''-} \notin \{w_{1}, w_{2}\}$. It is clear that ${N_{G_{1}}}({x^{''-}}) \cap {H} = \emptyset $, and at least two of $\{v, w_{1}, w_{2}\}$ cannot be adjacent to $x^{''-}$. It follows that $d_{G_{1}}(x^{''-}) \leqslant 2k + 3 -(k+1) - 2-1 = k-1$, a contradiction.

These contradictions complete our proof for $k\geqslant6$. We next discuss the case of $3\leqslant k\leqslant5$. Since $2k+1$ is odd, we only need to discuss the cases of $k=4$. When $k=4$, we have $n_{1}=11$. Since $|C|\geqslant6$, we consider the following four cases. It is worth noting that we have already proved the case when $R$ is an independent set in the preceding paragraph. Now, we first prove a simple claim.
\begin{clm}\label{lem4}
When $k=4$, if $R$ contains an edge $v_{0}v_{1}$ and $|C|\leqslant9$, then $N_{C}(v_{0})=N_{C}(v_{1})$.
\end{clm}
\begin{proof}
By contradiction. Let $Q=v_{0}v_{1}$. Suppose $N_{C}(v_{0})\neq N_{C}(v_{1})$, and without loss of generality, $|N_{C}(v_{0})-N_{C}(v_{1})|=1$. Since $G$ is a 4-regular graph, we have $|N_{C}(v_{0})\cup N_{C}(v_{1})|=4$ and $t(Q)=4$. Since $C$ is a longest cycle of $G_{1}$ containing $w_{1}$ and $w_{2}$, we have$|C| \geqslant |N_{C}(v_{0})\cup N_{C}(v_{1})|+2(t(Q)-2)+2\geqslant 10$, a contradiction.
\end{proof}

Case (a): $|V(C)|=9$. Let $C=u w_{1} v w_{2} z x_{1} x_{2} x_{3} x_{4}$. Then $R$ is an edge $v_{0}v_{1}$. By Claim~\ref{lem4} , we have ${N_C}(v_{0})={N_C}(v_{1})$. By assumption and symmetry, we have ${N_C}(R)=\{v, z, x_{3}\}$,
${N_C}(R)=\{v, z, x_{4}\}$, or ${N_C}(R)=\{x_{1}, x_{4}, v\}$. If ${N_C}(R)=\{x_{1}, x_{4}, v\}$, we have $x_{2}u\notin E(G_{1})$, otherwise, there is a Hamiltonian cycle $C'= v_{0} x_{1} z w_{2} v w_{1} u x_{2} x_{3} x_{4} v_{1}$. So we have ${d_{G_{1}}}(x_{2})\leqslant3$, a contradiction. The proofs for ${N_C}(R)=\{v, z, x_{3}\}$ and ${N_C}(R)=\{v, z, x_{4}\}$ are similar.

Case (b): $|V(C)|=8$. Let $C=u w_{1} v w_{2} z x_{1} x_{2} x_{3}$.

Subcase (b1): $R$ consists of an edge $v_{0}v_{1}$ and an isolated vertex $v_{2}$. By Claim~\ref{lem4} , we have ${N_C}(v_{0})={N_C}(v_{1})$. By assumption and symmetry, we have ${N_C}(v_{0})={N_C}(v_{1})=\{v, z, x_{3}\}$. Since $v\in {N_C}(v_{2})$, we have ${d_{G_{1}}}(v)\geqslant5$, a contradiction.

Subcase (b2): $R$ contains only one connected component $H$ which has three vertices. Since $G$ is a 4-regular graph, we have  ${d_C}(H)\geqslant 6$ and $|{N_C}(H)|\geqslant 3$. By assumption and symmetry, we have $\{v, z, x_{2}\}\subseteq {N_C}(H)$, $\{v, z, x_{3}\}\subseteq {N_C}(H)$, $\{u, z, x_{2}\}\subseteq {N_C}(H)$ or $\{v, x_{1}, x_{3}\}\subseteq {N_C}(H)$. If $\{v, z, x_{2}\}\subseteq {N_C}(H)$, we have $x_{1}, x_{3}\notin {N_C}(H)$. We claim $x_{1}x_{3}\notin E(G_{1})$, otherwise, there is a longer cycle $C'= H z w_{2} v w_{1} u x_{3} x_{1} x_{2}$ containing $w_{1}$ and $w_{2}$. Since $G$ is a 4-regular graph, we have $v\in{N_{G_{1}}}(x_{1})$ and $v\in{N_{G_{1}}}(x_{3})$, ${d_{G_{1}}}(v)\geqslant5$, a contradiction. The proofs for the other three cases are similar.

Case (c): $|V(C)|=7$. Let $C=u w_{1} v w_{2} z x_{1} x_{2}$. Clearly, there cannot have isolated vertex in $R$.

Subcase (c1): $R$ consists of two edges $v_{0}v_{1}$ and $v_{2}v_{3}$  that are not in the same component. By Claim~\ref{lem4} , we have $v\in {N_C}(v_{i}), i=0,1,2,3$, implying ${d_{G_{1}}}(v)\geqslant6$, a contradiction.

Subcase (c2): $R$ contains only one connected component $H$ which has four vertices. Since $G$ is a 4-regular graph, we have  ${d_C}(H)\geqslant 4$ and $|{N_C}(H)|\geqslant 2$. By assumption and symmetry, we have $\{v, x_{1}\}\subseteq {N_C}(H)$ or $\{z, x_{2}\}\subseteq {N_C}(H)$. If $\{v, x_{1}\}\subseteq {N_C}(H)$, we have $z, x_{2}\notin {N_C}(H)$. Since $G$ is a 4-regular graph, we have $vx_{2}, zx_{2}\in E(G_{1})$. So $v$  can accept at most 1 edge from $H$. Then ${N_C}(H)\geqslant 3$ and $u\in {N_C}(H)$. There is a longer cycle $C'= H u w_{1} v w_{2} z x_{2} x_{1}$ containing $w_{1}$ and $w_{2}$, a contradiction. The proof for $\{z, x_{2}\}\subseteq {N_C}(H)$ is similar.

Case (d): $|V(C)|=6$. Let $C=u w_{1} v w_{2} z x_{1}$. Clearly, there is no component $H$ in $R$ such that $|H|=1$ or $|H|=2$. So $R$ contains only one connected component $H$ which has five vertices. Obviously, there exist two consecutive vertices of $\{u, z, x_{1}\}$ which are adjacent to $H$, a contradiction.

Thus, we complete the proof of Theorem~\ref{thm3}.
\section{Proof of Theorem~\ref{thm5}}
Let $G$  be a 3-connected, $k$-regular graph on $n\leqslant 3k-6$ vertices, and let $P=uvz$ be a path of  $G$ such that $\{u, v, z\}$ is not a cut-set. After the operation in section 2, we have $|G_{1}|=n_{1} \leqslant 3k-4$. From Theorem~\ref{thm2}, we have $3k-6\geqslant 2k+1$, so, $k\geqslant 7$. Suppose $G_{1}$ is not Hamiltonian. Let $C$ be a longest cycle of $G_{1}$ containing $w_{1}$ and $w_{2}$, such that the number of components of $R=G_{1}-C$ is as small as possible. Let $r=|R|$ and $C=c_{1} c_{2}  \cdots c_{n_{1}-r}$. The subscripts of $c_{i}$ will be reduced modulo $n_{1}-r$ throughout. Obviously, we have $|C|=n_{1}-r \geqslant 7$.

When $R$ is an independent set, by the same proof as in section 3, we get $n_{1} \geqslant 3k$, which contradicts $n_{1} \leqslant 3k - 4$.
Thus in the following proof, we assume that there exists a component $H$ in $R$ such that $|H| \geqslant 2$. For a path $Q = {q_1} {q_2}  \cdots  {q_g}, g \geqslant 2$, in $H$, let $t(Q)$ denote the number of $C[c_{i}, c_{j}]$  such that $c_{i}$ is joined to one of $q_{1}$ and $q_{g}$, $c_{j}$ is joined to the other, and $e(\left\{ {{q_1}, {q_g}} \right\}, \{{c_{i + 1}}, {c_{i + 2}},  \cdots , {c_{j - 1}}\}) = 0$. We say that $Q$ satisfies the condition $(*)$ if $t(Q)\geqslant 2$,  ${N_{C}}(\left\{ {{q_1}, {q_g}} \right\}) \not\subset
 \left\{ {u, v, z} \right\}$ and there is a  $C[c_{i}, c_{j}]$ such that $u, v, z, {w_1}\ $and$ \ {w_2} \notin \left\{ {{c_{i + 1}}, {c_{i + 2}},  \cdots , {c_{j - 1}}} \right\}$. Now, let $H$ be the largest component of $R$ and $h=|H|$. Consider the following three cases.

$\bf Case \  1. $ $2 \leqslant h \leqslant k$.

By the same proof as in claim~\ref{lem1}, if $H$ is not Hamilton-connected, we have $n_{1}\geqslant (k-\frac{g}2-1)(g+1)+3$.
 Since $f(g)=(k-\frac{g}2-1)(g+1)+3$ is a concave function of $g$, $3\leqslant g\leqslant k$, we have $f(3) = 4k-7 > 3k-4$ and $f(k)=\frac{k^{2}}2-\frac{k}2+2  > 3k-4 $ when $k\geqslant 7$. Hence $f(g)> 3k-4$, a contradiction. Then $H$ is Hamilton-connected.

Let $M=\{r_{i}x_{i}\in E(G_{1}):i\in \{1,\ 2,\cdots, \ m\}, \ r_{i}\in V(H), \ x_{i}\in V(C)\}$ be a maximum matching in $E_{G_{1}}(H, C)$ and $m=|M|$, then $m\geqslant3$ since $G$ is 3-connected. In order to prove case 1, we next consider the following two subcases.

$\bf Subcase \  1.1. $ $k-2 \leqslant h \leqslant k$.

We claim $m = 3$. Otherwise, suppose $m \geqslant 4$, we must have $ n_{1} \geqslant|H| +|C| \geqslant h+m+(m-2)h+2\geqslant h+2+2h+4 =3h+6 \geqslant 3(k-2)+6=3k$, a contradiction.

 Let $x_{1},\ x_{2},\ x_{3}$ be in this order on $C$. Since $\{u, v, z\}$ is not a cut-set, without loss of generality, let $x_{3}\notin \{ u,\ w_1,\ v,\  w_2,\ z\} $ and $x_{3}^{-}\notin \{{w_1, w_2}\}$, $x_{3}^{+}\notin \{{w_1, w_2}\}$. Then we have either $C[x_{3}, x_{1}]$ or $C[x_{2}, x_{3}]$ which does not contain $w_1$ and $w_2$. When ${w_1, w_2}\notin C[x_{3}, x_{1}]$, since $C$ is a longest cycle of $G_{1}$ containing $w_{1}$ and $w_{2}$, we have
\[{N_{C}}(x_{3}^{-}) \cap \left[ {H \cup \left\{ x_{1}^{-}, x_{1}^{-2} \cdots , x_{1}^{-h} \right\} \cup x_{3}^{-} } \right] = \emptyset .\]

 And there are at least two of $\{v, w_{1}, w_{2}\}$ which can not be adjacent to $x_{3}^{-}$. It follows that \[{d_{G_{1}}}(x_{3}^{-})\leqslant 3k - 4 - 2h - 3 \leqslant 3k - 4 - 2(k - 2) - 3 = k-3,\]  this contradicts ${d_{G_{1}}}(x_{3}^{-})=k$. When ${w_1, w_2}\notin C[x_{2}, x_{3}]$, we have
\[{N_{C}}(x_{3}^{+}) \cap \left[ {H \cup \left\{ x_{2}^{+}, x_{2}^{+2} \cdots , x_{2}^{+h} \right\} \cup x_{3}^{+} } \right] = \emptyset .\]

 And there are at least two of $\{v, w_{1}, w_{2}\}$ which can not be adjacent to $x_{3}^{+}$. It follows that \[{d_{G_{1}}}(x_{3}^{+})\leqslant 3k - 4 - 2h - 3  \leqslant 3k - 4 - 2(k - 2) - 3 = k-3,\] this contradicts ${d_{G_{1}}}(x_{3}^{+})=k$.

$\bf Subcase \  1.2. $ $2 \leqslant h \leqslant k-3$.

For any $v\in V(H)$, since $2 \leqslant h\leqslant k-3$, we have ${N_{C}}({v}) \geqslant k-h+1\geqslant4$. By a similar proof as in Case 1.3, we have $n_{1} \geqslant \left| C \right| + \left| H \right| \geqslant k+3+(k-h-1)h.$
Put $g(h)=k+3+(k-h-1)h$. Since $g(h)$ is a concave function of $h$ with $g(2)= 3k-3=g(k-3)$. Hence $g(h)> 3k-4$, a contradiction.

$\bf Case \  2. $ $k+1 \leqslant h \leqslant 2k-7$.

 Let $Q = {q_1} {q_2}  \cdots  {q_g}$ be a path in $H$, which is chosen as long as possible such that $Q$ satisfies the condition $(*)$ (note that we are assuming 3-connectedness and that $\{u, v, z\}$ is not a cut-set). Put $A = {N_{C}}({q_1})$ and $B = {N_{C}}({q_g})$.

 \begin{clm}\label{lem5}
  $2\leqslant g\leqslant k-8$.
\end{clm}
\begin{proof}
Suppose that $g\geqslant k-7$. By the definition of the condition $(*)$, there is a  $C[c_{i}, c_{j}]$ such that $u, v, z, {w_1}\ and \ {w_2} \notin \left\{ {{c_{i + 1}}, {c_{i + 2}},  \cdots , {c_{j - 1}}} \right\}$. By the assumption of 3-connectedness and as $\{u, v, z\}$ is not a cut-set, we have  either $c_{i}^{-} \notin \{w_{1}, w_{2}\}$ or $c_{j}^{+} \notin \{w_{1}, w_{2}\}$. Without loss of generality, let $c_{i}^{-} \notin \{w_{1}, w_{2}\}$. We have ${N_{C}}(c_{i}^{-}) \cap \{H \cup \{{ c_{i}^{-}, { c_{j-1},c_{j-2},\cdots,c_{j-g}}}\}\} = \emptyset $, and at least two of $\{v, w_{1}, w_{2}\}$ cannot be adjacent to $c_{i}^{-}$. Thus
\[{d_{G_{1}}}(c_{i}^{-}) \leqslant 3k-4 - \left[ {h + g+2+1} \right] \leqslant3k-4-(k+1)-(k-7)-2-1  \leqslant k - 1.\]
This contradicts that ${d_{G_{1}}}(c_{i}^{-})=k$.
\end{proof}

 \begin{clm}\label{lem6}
$Q$ is a maximal path in $H$ satisfying $(*)$.
\end{clm}
\begin{proof}
Suppose that $Q$ is not a maximal path in $H$.

Let $Q' ={b_1} {b_2}  \cdots  {b_s} {q_1} {q_2}  \cdots {q_g} {q_{g+1}}  \cdots {q_e}$ be a maximal path in $H$ containing $Q$. Without loss of generality, we assume $s \geqslant 1$. From the definition of $Q$, we have $N_{C}(b_{1})\leqslant3$, which implies $N_{H}({b_1}) \geqslant k-3$.

 When $N_{H}(b_{1})\cap\{q_{2},q_{3},\cdots,q_{e}\}\neq \emptyset$, one of the  following three cases occurs.
 \begin{enumerate}[(1).]
   \item $N_{H}(b_{1})\cap\{q_{g+1},q_{g+2},\cdots,q_{e}\}= \emptyset$.
   \item $N_{H}(b_{1})\cap\{q_{2},q_{3},\cdots,q_{g-1}\}= \emptyset$.
   \item $N_{H}(b_{1})\cap\{q_{g+1},q_{g+2},\cdots,q_{e}\}\neq \emptyset$ and $N_{H}(b_{1})\cap\{q_{2},q_{3},\cdots,q_{g-1}\}\neq \emptyset$.
 \end{enumerate}

In (1), set $i=min\{j\geqslant2: b_{1}q_{j}\in E(G_{1})\}$. Let $Q'' ={q_1} {b_s} {b_{s-1}} \cdots {b_1} {q_i} {q_{i+1}}  \cdots  {q_g}$. In (2), set $j=max\{d\geqslant g: b_{1}q_{d}\in E(G_{1})\}$. Let $Q'' ={q_1} {b_s} {b_{s-1}} \cdots {b_1} {q_j} {q_{j-1}}  \cdots {q_g}$. In both (1) and (2), since $N_{H}(b_{1})\cup\{b_{1}\} \subseteq Q''$, then $g\geqslant |V(Q'')|\geqslant k-2$, a contradiction to Claim \ref{lem5}.

In (3), set $$l_{1}=min\{j:2\leqslant j\leqslant g-1, b_{1}q_{j}\in E(G_{1})\},$$  $$l_{2}=max\{j: 2\leqslant j\leqslant g-1, b_{1}q_{j}\in E(G_{1})\},$$ and $$l_{3}=max\{j: b_{1}q_{j}\in E(G_{1})\}.$$ Let $Q^{*} ={q_1} {b_s} {b_{s-1}} \cdots {b_1} {q_{l_{1}}} {q_{l_{1}+1}} \cdots  {q_g}$ and $Q^{**} ={q_1} {q_2} \cdots {q_{l_{2}}} {b_1} {q_{l_{3}}} {q_{l_{3}-1}} \cdots {q_g}$. Because $g\geqslant max\{|V(Q^{*})|, |V(Q^{**})|\}$, we have $l_{1}-2\geqslant s$ and $g-1-l_{2}\geqslant l_{3}-g+1$, which implies $g\geqslant (l_{2}-l_{1}+1)+s+l_{3}-g+1+2$. So
\[\begin{array}{l}
g \geqslant 1+|\{b_1,\cdots, {b_s}\}\cup \{q_1, {q_g}\}\cup \{{q_{l_{1}}}, {q_{l_{1}+1}}, \cdots , {q_{l_{2}}}\}\cup \{{q_{g+1}}, \cdots , {q_{l_{3}}}\}| \\
\ \ \ \geqslant 1+|N_{H}(b_{1})\cup\{b_{1}\}|\geqslant k-1,
\end{array}\]a contradiction to Claim \ref{lem5}.

A similar argument holds if $N_{H}(q_{e})\cap\{b_{1},b_{2},\cdots,b_{s},q_{1},q_{2},\cdots,q_{g-1}\}\neq \emptyset$. Thus it is enough for completing our proof to discuss the following two cases.
\begin{enumerate}[(a).]
  \item $e>g,~N_{H}(b_{1})\subseteq \{b_{2},b_{3},\cdots,b_{s},q_{1}\} $ and $N_{H}(q_{e})\subseteq \{q_{g},q_{{g+1}},\cdots, q_{{e-1}}\} $
  \item $e=g,~N_{H}(b_{1})\subseteq \{b_{2},b_{3},\cdots,b_{s},q_{1}\} $ .
\end{enumerate}

In case (a), since $|N_{C}(b_{1})|\leqslant 3$, we have $|\{b_1,\cdots, {b_s}\},q_1|\geqslant |N_{H}(b_{1})\cup\{b_{1}\}|\geqslant k-2$. Similarly $|\{q_g,\cdots, {q_{e-1}}\},q_e|\geqslant |N_{H}(q_{e})\cup\{q_{e}\}|\geqslant k-2$. Then we have $|H|\geqslant |N_{H}(b_{1})\cup\{b_{1}\}|+|N_{H}(q_{e})\cup\{q_{e}\}|\geqslant 2k-4$. This contradicts $k+1 \leqslant h \leqslant 2k-7$.

In case (b), denote $q_{1}=b_{s+1}$. Let $i=max\{j:b_{1}b_{j}\in E(G_{1})\}$. We claim that $N_{C}(b_{l})=\{v\}$ and $N_{H}(b_{l})\cap \{q_{2},q_{3},\cdots,q_{g}\}=\emptyset$, for any $1\leqslant l\leqslant i-1$.
Otherwise, let $d=min\{j:j>l ~and ~b_{1}b_{j}\in E(G_{1})\}$.
 When $|N_{C}(b_{l})|\geqslant2$, either ${b_l} {b_{l-1}} \cdots {b_1} {b_d} {b_{d+1}} \cdots {b_s} {q_1}$ or ${b_l} {b_{l-1}} \cdots {b_1} {b_d} {b_{d+1}} \cdots {b_s} {q_1} {q_2} \cdots {q_g}$ is a path that satisfies $(*)$ and is longer than $Q$, a contradiction. When $q_{f}\in N_{H}(b_{l})$ for some $2\leqslant f\leqslant g$, then $${Q'''=q_1} {b_s} {b_{s-1}} \cdots {b_d} {b_1} {b_2} \cdots {b_{l}} {q_f} {q_{f+1}} \cdots {q_g}$$ is the path that satisfies $(*)$ and is longer than $Q$, a contradiction.

 Since $|Q'|\geqslant |N_{H}(b_{1})|+|N_{H}(q_{g})|+2-1$, we have
 \[\begin{array}{l}
n_{1}\geqslant |Q'|+|C| \\
 \ \ \ \ \geqslant |N_{H}(b_{1})|+k-|B|+1+|B|+(|B|-1+g) \\
 \ \ \ \ \geqslant|N_{H}(b_{1})|+k+|B|+g\\
 \ \ \ \ \geqslant 2k-1+|B|+g.
\end{array}\]
So $g\leqslant 3k-4-2k+1-|B|< k-|B|=|N_{H}(q_{g})|$. It follows that $N_{H}(q_{g})\cap \{{b_1}, {b_2},  \cdots , {b_s}\}\neq \emptyset $. Then $|N_{C}(q_{2})|\leqslant 3$ as there is a path in $H$ with at least $g + 1$ vertices connecting $q_{1}$ and $q_{2}$, a contradiction. Thus we have $|H|\geqslant |N_{H}(b_{1})|+|N_{H}(q_{2})|+1\geqslant k-1+k-3+1\geqslant 2k-3$, a contradiction.

\end{proof}

\begin{clm}\label{lem7}
 $t(Q)\geqslant3$.
\end{clm}
\begin{proof}
From Claim \ref{lem5} and Claim \ref{lem6}, we have $|A|\geqslant 9$ and $|B|\geqslant 9$. Suppose $t(Q)=2$. There is only one case: $A=\left\{ {{c_{{i_1}}}, {c_{{i_2}}},  \cdots , {c_{{i_s}}}} \right\}$ and $B=\left\{ {{c_{{j_1}}}, {c_{{j_2}}},  \cdots , {c_{{j_l}}}} \right\}$ such that $s\geqslant 9, l\geqslant 9$ and $\left\{ {{c_{{i_1}}}, {c_{{i_2}}},  \cdots , {c_{{i_s}}}} \right\} \cap \left\{ {{c_{{j_1}}}, {c_{{j_2}}},  \cdots , {c_{{j_l}}}} \right\} = \emptyset $.

There is at least one segment of $C [c_{j_{l}}, c_{i_1}]$ and $C [c_{i_s}, c_{j_{1}}]$ which does not contain $w_{1}, w_{2}$. Without loss of generality, let $w_{1}, w_{2} \notin C [c_{j_{l}}, c_{i_1}]$, then there exists some ${c_z} \in {A^ + }$ satisfying ${N_{C}}({c_z}) \cap \left[ {H \cup\left\{ {{w_1}, {w_2}, {c_{z}}}, c_{j_{l+1}},c_{j_{l+2}},\cdots,c_{j_{l+g}}\right\}} \right] = \emptyset $. And for any $C [c_{j_{f}}, c_{j_{f+1}}]$ such that $C [c_{j_{f}}, c_{j_{f+1}}]\cap \{w_{1}, w_{2}\}=\emptyset$, $f=1,2,\cdots,l-1$, we have ${N_{C}}({c_z}) \cap \{ c_{j_{f}+1}, c_{j_{f}+2} \} = \emptyset $. This implies \[{d_{G_{1}}}({c_z}) \leqslant 3k -4 - \left[ {h+g + 2(l - 1-2) + 3} \right] \leqslant3k-1-2(g+l)\leqslant k-3,\]
a contradiction(note that $g+l\geqslant k+1$).

\end{proof}

In fact, we have $t(Q)\geqslant4$. Otherwise, suppose $t(Q)=3$, We only have to consider one case: $A\cap B=\{c_{m}\}$, $A\setminus\{c_{m}\}=\left\{ {{c_{{i_1}}}, {c_{{i_2}}},  \cdots , {c_{{i_s}}}} \right\}$ and $B\setminus\{c_{m}\}=\left\{ {{c_{{j_1}}}, {c_{{j_2}}},  \cdots , {c_{{j_l}}}} \right\}$ such that $s\geqslant 8, l\geqslant 8$ and $\left\{ {{c_{{i_1}}}, {c_{{i_2}}},  \cdots , {c_{{i_s}}}} \right\} \cap \left\{ {{c_{{j_1}}}, {c_{{j_2}}},  \cdots , {c_{{j_l}}}} \right\} = \emptyset $. It is worth noting that in this case, $g+l\geqslant k$. By a similar proof as in Claim \ref{lem7}, we have that there exists some ${c_z} \in {A^ + }$ or ${c_z} \in {B^ + }$ such that ${d_{G_{1}}}({c_z}) \leqslant 3k -4 - \left[ {h+g + 2(l - 1-2) + 3} \right] \leqslant3k-1-2(g+l)\leqslant k-1$, a contradiction.
Since ${N_{C}}(\{q_1,\ q_g\})\cup ({N_{C}}(\{q_1,\ q_g\}))^{+}\cup H \subseteq V(G_{1}) $, we have
\[\begin{array}{l}
n_{1} \geqslant \left| C \right| + \left| H \right| \\
\ \ \ \ \geqslant \left| H \right|+2|{N_{C}}(\{q_1,\ q_g\})|+(t(Q)-2)(g-1)\\
 \ \ \ \ \geqslant h+2(k-g+1)+(t(Q)-2)(g-1)\\
\ \ \ \ \geqslant3k+1+(t(Q)-4)(g-1).
\end{array}\]
Because of $n_{1}\leqslant 3k-4$, we have $-5\geqslant(t(Q)-4)(g-1)$, a contradiction to $t(Q)\geqslant4$ and $g\geqslant2$.

$\bf Case \  3. $ $h \geqslant 2k-6$.

Since $G_{1}$ is a 3-connected graph and $\{u, v, z\}$ is not a cut-set, there exists a vertex $x{'} \in {N_{C}}(H)$, such that $x^{'-} \notin \{w_{1}, w_{2}\}$. It is clear that ${N_{G_{1}}}({x^{'-}}) \cap {H} = \emptyset $, and at least two of $\{v, w_{1}, w_{2}\}$ cannot be adjacent to $x^{'-}$. It follows that \[d_{G_{1}}(x^{'-}) \leqslant 3k-4-(2k-6)-2-1 =k-1,\] a contradiction.

Thus, we complete the proof of Theorem~\ref{thm5}.

\section{Conclusion}
In this paper, we characterize a class of graphs that illustrate the sharpness of the bound $2k$ in Theorem~\ref{thm2}. By excluding these particular graphs, we are able to enhance the result and establish that the bound is in fact $2k+1$ for 2-connected, $k$-regular graphs and $3k-6$ for 3-connected, $k$-regular graphs.
The problem of regular 3-connected $2$-path Hamiltonian graphs with $n$ vertices remains intriguing whenever $3k-5\leqslant n\leqslant 3k-1$. Naturally, we may inquire whether any interesting properties can be observed in regular $m$-connected graphs for $m\geqslant 4$? The resolution of the aforementioned question is necessarily relevant to the study of the Hamiltonicity and edge-Hamiltonicity of regular graphs. However, it should be noted that the existence of an $L$-graph (as illustrated in~\cite{Hao Li}) presents a counterexample, showcasing higher connectivity that prevents the realization of edge-Hamiltonicity in regular graphs. As a result, we establish $3k-1$ as an upper bound for achieving the optimum outcome.



\begin{thebibliography}{99}

\bibitem{bondy} J.A. Bondy, U.S.R. Murty, Graph theory with applications, Macmillan, London, 1976.
\bibitem{BH}B. Bollob\'{a}s, A.M. Hobbs, Hamiltonian cycles in regular graphs, in: Advances in Graph Theory, North Holland, Amsterdam, 1978.
\bibitem{BK} J.A. Bondy, M. Kouider, Hamilton cycles in regular 2-connected graphs, Journal of Combinatorial Theory. Series B 44 (1988) 177--186.
\bibitem{EH} P. Erd\"{o}s, A.M. Hobbs, Hamiltonian cycles in regular graphs of moderate degree, Journal of Combinatorial Theory. Series B 23 (1977) 139--142.
\bibitem{H}  R. H\"{a}ggkvist, Unsolved problems, in: Proceedings of the Fifth Hungarian Colloquim on Combinatorics, 1976.
\bibitem{J} B. Jackson, Hamilton cycles in regular 2-connected graphs, J. Combin. Theory, Ser. B 29 (1980) 27--46.

\bibitem{Bill} B. Jackson, H. Li, Y. Zhu, Dominating cycles in regular 3-connected graphs, Discrete Mathematics 102 (1991) 163--176.
\bibitem{Hudson V. Kronk} H. V. Kronk, A note on $k$-path Hamiltonian graphs, Journal of Combinatorial Theory 7 (1969) 104--106.

\bibitem{Kuhn} D. K\"{u}hn, A. Lo, D. Osthus, K. Staden, Solution to a problem of Bollob\'{a}s and H\"{a}ggkvist on Hamilton cycles in regular graphs, Journal of Combinatorial Theory. Series B 121 (2016) 85--145.
\bibitem{L}H. Li, Hamiltonian cycles and circumferences in graphs, Ph.D. Dissertation, Institute of Systems Science, Academia Sinica 1986.
\bibitem{Hao Li} H. Li, Edge-Hamiltonian property in regular 2-connected graphs, Discrete Mathematics 82 (1990) 25--34.
\bibitem{Xia Li} X. Li, W. Yang, On regular 2-connected 2-path Hamiltonian graphs, Discrete Mathematics 346 (2023) 113468.
\bibitem{MWY} R. McCarty, Y. Wang, X. Yu, 7-Connected graphs are 4-ordered, Journal of Combinatorial Theory. Series B 141 (2020) 115--135.
\bibitem{NS}  L. Ng, M. Schultz, $k$-ordered Hamiltonian graphs, Journal of Graph Theory 24 (1) (1997) 45--57.
\bibitem{Woodall} D.R. Woodall, The binding number of a graph and its Anderson number, Journal of Combinatorial Theory. Series B 15 (1973) 225--255.
\bibitem{ZLY}Y. Zhu, Z. Liu, Z. Yu, 2-connected $k$-regular graphs on at most $3k+3$ vertices to be Hamiltonian, Journal of Systems Science and Mathematical Sciences 6 (1986) 36--49 and 136--145.



























\end{thebibliography}
\end{document}